\documentclass[11pt]{amsart}

\usepackage{amsmath}
\usepackage{fullpage}
\usepackage{xspace}
\usepackage[psamsfonts]{amssymb}
\usepackage[latin1]{inputenc}
\usepackage{graphicx,color}
\usepackage[curve]{xypic} 
\usepackage{hyperref}
\usepackage{graphicx}

\usepackage{amsmath}%
\usepackage{amsthm}%
\usepackage{amscd}
\usepackage{amsfonts}%
\usepackage{amssymb}%
\usepackage{graphicx}

\usepackage{mathrsfs}

\usepackage{tikz}
\usetikzlibrary{matrix,arrows}

\usepackage{tikz-cd}

\newtheorem{theorem}{Theorem}[section]
\theoremstyle{plain}

\newtheorem{conjecture}[theorem]{Conjecture}

\newtheorem{lemma}[theorem]{Lemma}

\theoremstyle{remark}

\numberwithin{equation}{section}

\newcommand{\Jcal}{\mathscr{J}}

\newcommand{\Ncal}{\mathscr{N}}

\newcommand{\Pro}{\mathbb{P}}

\newcommand{\Z}{\mathbb{Z}}
\newcommand{\C}{\mathbb{C}}

\newcommand{\F}{\mathbb{F}}
\newcommand{\Q}{\mathbb{Q}}

\newcommand{\rk}{\mathrm{rk}\,}

\newcommand{\MW}{\mathrm{MW}}

  \DeclareFontFamily{U}{wncy}{}
    \DeclareFontShape{U}{wncy}{m}{n}{<->wncyr10}{}
    \DeclareSymbolFont{mcy}{U}{wncy}{m}{n}
    \DeclareMathSymbol{\Sha}{\mathord}{mcy}{"58}

\begin{document}
\title[]{On the fibres of an elliptic surface where the rank does not jump}

\author{Jerson Caro}
\address{ Departamento de Matem\'aticas,
Pontificia Universidad Cat\'olica de Chile.
Facultad de Matem\'aticas,
4860 Av.\ Vicu\~na Mackenna,
Macul, RM, Chile}
\email[J. Caro]{jocaro@uc.cl}%

\author{Hector Pasten}
\address{ Departamento de Matem\'aticas,
Pontificia Universidad Cat\'olica de Chile.
Facultad de Matem\'aticas,
4860 Av.\ Vicu\~na Mackenna,
Macul, RM, Chile}
\email[H. Pasten]{hector.pasten@mat.uc.cl}%

\thanks{J.C. was supported by ANID Doctorado Nacional 21190304. H.P. was supported by ANID (ex CONICYT) FONDECYT Regular grant 1190442 from Chile. }

\date{\today}
\subjclass[2010]{Primary: 14J27; Secondary: 11G05, 11G40} %
\keywords{Elliptic surface, rank, parity conjecture}%

\begin{abstract} For a non-constant elliptic surface over $\mathbb{P}^1$ defined over $\mathbb{Q}$,  it is a result of Silverman that the Mordell--Weil rank of the fibres is at least the rank of the group of sections, up to finitely many fibres. If the elliptic surface is non-isotrivial one expects that this bound is an equality for infinitely many fibres, although no example is known unconditionally. Under the Bunyakovsky conjecture, such an example has been constructed by Neumann and Setzer. In this note we show that the Legendre elliptic surface has the desired property, conditional on the existence of infinitely many Mersenne primes.
\end{abstract}

\maketitle



\setcounter{secnumdepth}{2}


\section{Introduction}

Let $\pi:X\to \Pro^1$ be a non-constant elliptic surface (with section) defined over $\Q$ and let $\MW(X,\pi)$ be its group of sections defined over $\Q$. This group is finitely generated by the Lang-N\'eron theorem \cite{LN}. By a theorem of Silverman \cite{SilvermanSp}, for all but finitely many $b\in \Pro^1(\Q)$ the fibre $X_b$ is an elliptic curve with $\rk X_b(\Q)\ge \rk \MW(X,\pi)$ and one can ask when is this inequality strict and when is it an equality. Thus, let us define
$$
\Jcal(X,\pi)=\{b\in \Pro^1(\Q) : X_b\mbox{ is an elliptic curve with }\rk X_b(\Q)>\rk \MW(X,\pi)\}
$$
and
$$
\Ncal(X,\pi)=\{b\in \Pro^1(\Q) : X_b\mbox{ is an elliptic curve with }\rk X_b(\Q)=\rk \MW(X,\pi)\}.
$$
These are the sets of points $b\in \Pro^1(\Q)$ where the rank jumps and where it does not jump respectively. The question is whether these sets are infinite.

There is a considerable body of work addressing this problem for $\Jcal(X,\pi)$ and we refer the reader to \cite{Salgado} and the references therein. 

On the other hand, much less is known for $\Ncal(X,\pi)$. In \cite{CasselsSchinzel} Cassels and Schinzel produced an example where $\Ncal(X,\pi)$ is empty, conditional on a conjecture of Selmer.  Regarding the infinitude of $\Ncal(X,\pi)$, there are plenty of results where this set is infinite for quadratic twists families, see for instance \cite{Li} for a survey on recent results. The example of Cassels and Schinzel, as well as quadratic twist families, are isotrivial.

In the non-isotrivial case one expects:
\begin{conjecture}\label{ConjMain} If $\pi:X\to \Pro^1$ is non-isotrivial then both $\Jcal(X,\pi)$ and $\Ncal(X,\pi)$ are infinite.
\end{conjecture} 
A heuristic for this conjecture is implicit in Appendix A of \cite{CCH}. We revisit this heuristic in Section \ref{SecHeuristic1} for the convenience of the reader.

 Despite the theoretical support for Conjecture \ref{ConjMain}, no example of non-isotrivial elliptic surface $\pi:X\to \Pro^1$ with an infinite $\Ncal(X,\pi)$ is known. 
 
 Let us consider the case when $\rk \MW(X,\pi)=0$. Such elliptic surfaces over $\Q$ are abundant, see for instance \cite{Kloosterman}. Regarding $\Ncal(X,\pi)$, as a very special case of Conjecture \ref{ConjMain} we have:
 
\begin{conjecture}\label{ConjRk0} Let $\pi:X\to \Pro^1$ be a non-isotrivial elliptic surface defined over $\Q$ with the property that $\rk \MW(X,\pi)=0$. Then there are infinitely many $b\in \Pro^1(\Q)$ with $\rk X_b(\Q)=0$. That is, $\Ncal(X,\pi)$ is infinite.
\end{conjecture}

This problem remains open and, to the best of our knowledge, not a single example of non-isotrivial elliptic surface with an infinite $\Ncal(X,\pi)$ is know. Nevertheless, let us consider the elliptic surface $\pi:S\to \Pro^1$ defined by the Weierstrass equation
$$
y^2 = x^3 + tx^2-16x
$$
where $t$ is an affine coordinate on $\Pro^1$. This elliptic surface is non-isotrivial and it was first studied by Neumann \cite{Neumann} and Setzer \cite{Setzer}, see also \cite{SteinWatkins}. If $p$ is a prime of the form $p=b^2+64$ for an integer $b\equiv 3\bmod 4$ then $S_b(\Q)\simeq \Z/2\Z$ and in particular $S_b$ has rank $0$ over $\Q$. A well-known conjecture of Bunyakovsky predicts that there are infinitely many primes $p$ of this form, so, one gets a conditional example where Conjecture \ref{ConjRk0} holds.

In this note we consider instead the Legendre elliptic surface $\pi:Y\to \Pro^1$ defined by the Weierstrass equation
$$
y^2 = x(x+1)(x+t).
$$
This elliptic surface has $\rk\MW(Y,\pi)=0$, see Lemma \ref{LemmaYrk0}. To state our main result, let us recall that a \emph{Mersenne prime} is a prime number $p$ of the form $p=2^q-1$ with $q$ a prime.  It is conjectured that there are infinitely many of them, see Section \ref{SecHeuristic2} for details. We prove:
\begin{theorem}\label{ThmMain1} Let $q\ge 5$ be a prime such that $p=2^q-1$ is a Mersenne prime. Then the elliptic curve $E_q$ defined by $y^2=x(x+1)(x+2^q)$ has $\rk E_q(\Q)=0$.

In particular, if there are infinitely Mersenne primes, then  the (non-isotrivial) Legendre elliptic surface $\pi:Y\to \Pro^1$  has the property that $\Ncal(Y,\pi)$ is infinite. 
\end{theorem}

The proof of \ref{ThmMain1} has two main ingredients. First, using a descent bound we show that $\rk E_q(\Q)\le 1$, which falls short of proving the result. However, using known cases of the parity conjecture due to Monsky \cite{Monsky} as well as some control on Shafarevich--Tate groups we deduce that $\rk E_q(\Q)=1$ is not possible. For this strategy to work we need to carefully analyze the reduction type of $E_q$ at the primes $2$ and $p$.

Finally, let us mention a somewhat unexpected motivation for Conjecture \ref{ConjRk0}. Although we have only discussed elliptic surfaces over $\Q$, Conjecture \ref{ConjRk0} can also be formulated over number fields.  In \cite{Pasten} it is shown that this conjecture over number fields implies that for every number field $K$ the analogue of Hilbert's tenth problem for the ring of integers $O_K$ is undecidable.


\section{Preliminaries}

\subsection{The Legendre elliptic surface}\label{SecLegendre} Let $t$ be the affine coordinate on $\Pro^1$. The \textit{Legendre elliptic surface} is the (relatively minimal) elliptic surface $\pi:Y\to\Pro^1$ defined over $\Q$ by the affine Weierstrass equation 
$$
y^2=x(x+1)(x+t).
$$
The following lemma is a direct computation.

\begin{lemma} The Legendre elliptic surface is non-isotrivial and rational. It has $3$ singular fibres: at $t=0$ of Kodaira type $I_2$, at $t=-1$ of Kodaira type $I_2$, and at $t=\infty$ of Kodaira type $I_2^*$.
\end{lemma}

Using this we get

\begin{lemma}\label{LemmaYrk0} We have $\rk \MW(Y,\pi)=0$. In fact, the group of sections over $\C$ is formed by the $2$-torsion sections.
\end{lemma}
\begin{proof} Let us base change to $\C$. Since $\pi:Y_\C\to \Pro^1$ is a rational elliptic surface with singular fibres of types $I_2$, $I_2$, and $I_2^*$, its group of sections is given in the entry 71 of the Main Theorem of \cite{OguisoShioda}.
\end{proof}

We refer the reader to \cite{Ulmer} for a detailed study of this elliptic surface.

\subsection{Bounds for the rank} The following result is a more precise version of the bound provided by Proposition 1.3 in \cite{CaroPasten}; the argument is a variation of the proof of Lemma 3.1 in \emph{loc. cit.}
\begin{theorem}\label{ThmRankBound}
Let $E$ be an elliptic curve over $\Q$ admitting a $2$-isogeny $\theta:E\to E'$ over $\Q$. Let $\alpha$ and $\mu$ be the number of places of additive and of multiplicative reduction of $E$ respectively. Then 
$$
\rk E(\Q)\le   2\alpha +\mu-1.
$$ 
If equality holds, then the $2$-primary part of the Shafarevich--Tate group $\Sha(E)$ is trivial.
\end{theorem}
\begin{proof} The claimed bound is precisely Proposition 1.3 in \cite{CaroPasten}. Suppose that equality holds. We will show that the $2$-torsion part of $\Sha(E)$ is trivial, which is sufficient.

Let $S_2(E)$ be the $2$-Selmer group of $E$. Then we have the exact sequence
$$
0\to E(\Q)/2E(\Q)\to S_2(E)\to \Sha(E)[2]\to 0
$$
and we see that it suffices to show
\begin{equation}\label{EqnEnough}
2\alpha+\mu-1 + \dim_{\F_2}E(\Q)[2]\ge \dim_{\F_2} S_2(E)
\end{equation}
because $\dim_{\F_2}E(\Q)/2E(\Q) = \rk E(\Q) + \dim_{\F_2}E(\Q)[2]=2\alpha+\mu-1 + \dim_{\F_2}E(\Q)[2]$.

Let $\theta:E\to E'$ be a rational $2$-isogeny and let $\theta':E'\to E$ be its dual. For the corresponding Selmer groups $S_\theta(E)$ and $S_{\theta'}(E)$ we have (cf. Theorem 2.2 in \cite{CaroPasten})
\begin{equation}\label{Eqns1s2}
\dim_{\F_2} S_{\theta}(E) + \dim_{\F_2} S_{\theta'}(E') \le 2\alpha + \mu+1.
\end{equation}
These Selmer groups are related to $S_2(E)$ via the exact sequence:
\begin{equation}\label{EqnLong}
0\to E'(\Q)[\theta']/\theta(E(\Q)[2])\to S_{\theta}(E)\to S_2(E)\to S_{\theta'}(E'),
\end{equation}
cf. Lemma 6.1 in \cite{SchSto} (note that although this is only claimed for odd primes $p$ in \emph{loc. cit.}, the hypothesis $p>2$ is not really needed at this point.)

Let us consider two cases. First, let us assume that $E(\Q)[2]\simeq \Z/2\Z$. Then \eqref{Eqns1s2} and \eqref{EqnLong} give
$$ 
\dim_{\F_2} S_2(E)\le 2\alpha + \mu+1 - \dim_{\F_2} E'(\Q)[\theta']/\theta(E(\Q)[2]) =2\alpha + \mu
$$
where $\theta(E(\Q)[2])=(0)$ because $E(\Q)[2]\simeq \Z/2\Z$. This proves \eqref{EqnEnough} when $E(\Q)[2]\simeq \Z/2\Z$. Now let us assume $E(\Q)[2]\simeq (\Z/2\Z)^2$. Then \eqref{Eqns1s2} and \eqref{EqnLong} directly give
$$ 
\dim_{\F_2} S_2(E)\le 2\alpha + \mu+1
$$
which proves \eqref{EqnEnough} in this case.
\end{proof}

\subsection{Root numbers} The root number $w(E)$ of an elliptic curve $E$ over $\Q$ is the sign of the functional equation of $E$. Furthermore, one can define the local root number $w_p(E)$ at a prime $p$, and these local root numbers are related to $w(E)$ via the formula
$$w(E)=-\prod_{p}w_p(E).$$
When $E$ is semi-stable we have a complete characterization for the local root numbers:
$$
w_p(E)=\begin{cases}
1 & \mbox{ if $E$ has good or non-split multiplicative reduction at $p$,}\\
-1 & \mbox{ if $E$ has split multiplicative reduction at $p$.}
\end{cases}
$$

\subsection{The parity conjecture} An immediate consequence of Birch and Swinnerton-Dyer conjecture is:
\begin{conjecture}[Parity Conjecture]\label{Parity conjecture}
Let $E$ be an elliptic curve over $\Q$, then $(-1)^{\rk E(\Q)}=w(E)$.
\end{conjecture}
The following is a direct consequence of Theorem 1.5 in \cite{Monsky} 
\begin{lemma}\label{LemmaMonsky}
Let $E$ be an elliptic curve over $\Q$. If the $2$-primary part of $\Sha(E)$ is finite, then $(-1)^{\rk E(\Q)}=w(E)$.
\end{lemma}


\section{Proof of the main result}

\subsection{Reduction types} 

\begin{lemma} \label{Lemma reduction} Let $p=2^q-1$ be a Mersenne prime with $q\ge 5$, and let $E_q$ be the elliptic curve defined by $y^2=x(x+1)(x+2^q)$. Then $E_q$ has split multiplicative reduction at $2$, non-split multiplicative reduction at $p$, and good reduction at all other primes.
\end{lemma}
\begin{proof} Lemma 1 in \cite{DiamondKramer} implies that $E_q$ has multiplicative reduction at $2$ and $p$, and good reduction at every other prime. The reduction of $E_q$ modulo $p$ is defined by the equation
$$
y^2=x(x+1)^2.
$$
Replacing $x$ by $x-1$ we get the model $y^2=x^3-x^2$ with the singular point at $(x,y)=(0,0)$. Thus, $E_q$ has split multiplicative reduction at $p$ if and only if $\sqrt{-1}\in \F_p$. Since $$p= 2^q-1\equiv -1 \bmod 4,$$ we conclude that $E_q$ has non-split multiplicative reduction at $p$. On the other hand, the proof of Lemma 2 in \cite{DiamondKramer} yields the following minimal equation for $E_q$ over $\Q_2$: 
$$
y^2+xy=x^3+2^{q-2}x+ 2^{q-4}x.
$$
As $q\ge 5$, the reduction modulo $2$ is given by the equation $y^2+xy=x^3$, with the singular point at $(x,y)=(0,0)$. This singular point is a nodal singularity with tangent slopes $0$ and $1$, both defined over $\F_2$. As a consequence, $E_q$ has split multiplicative reduction at $2$.
\end{proof}

\subsection{Bounding the rank}

\begin{proof}[Proof of Theorem \ref{ThmMain1}] By Lemma \ref{Lemma reduction} we have that the bad primes for $E_q$ are $2$ with split multiplicative reduction, and $p$ with non-split multiplicative reduction. Therefore, $w(E)=1$ since $w_2(E)=-1$ and $w_p(E)=1$. 

Let us apply Theorem \ref{ThmRankBound}.  We obtain $\rk E_q(\Q)\leq 1$ and we claim that, in fact, $\rk E_q(\Q)=0$. For the sake of contradiction, assume that $\rk E_q(\Q)= 1$. Then Theorem \ref{ThmRankBound} implies that the $2$-primary part of  $\Sha(E_q)$ is trivial, hence, by Lemma \ref{LemmaMonsky} we would obtain $w(E)=(-1)^{\rk E_q(\Q)}=-1$, which contradicts the fact that $w(E)=1$.
\end{proof}


\section{Heuristics}

\subsection{On Conjecture \ref{ConjMain}}\label{SecHeuristic1}

Let $\pi:X\to\Pro^1$ be a non-isotrivial elliptic surface defined over $\Q$ and let $R=\rk \MW(X,\pi)$.

Following the terminology in \cite{CCH}, the \emph{density conjecture} of Silverman  \cite{SilvermanConj} asserts that 
$$
\rk X_b(\Q)\in\{R,R+1\}
$$ 
for all $b\in \Pro^1(\Q)$ outside a set of density $0$ in $\Pro^1(\Q)$. 

In \cite{Helfgott}, Helfgott proved that, under some conjectures in analytic number theory, the average value of the root numbers $w(X_b)$ as $b$ varies in $\Pro^1(\Q)$ (ordered by height) exists and it is expressed in terms of local densities. As pointed out in p. 728 of Appendix A in  \cite{CCH}, the conjectural value for this average of roots numbers lies strictly between $-1$ and $1$ in the non-isotrivial case.

Therefore, under the previous conjectures, we deduce Conjecture \ref{ConjMain} in a strong form: both sets $\Jcal(X,\pi)$ and $\Ncal(X,\pi)$ should have positive density in $\Pro^1(\Q)$.

\subsection{On Mersenne primes}\label{SecHeuristic2} Let us recall the following folklore conjecture

\begin{conjecture}\label{ConjMersenne} There are infinitely many Mersenne primes.
\end{conjecture}

Mersenne primes provide a way to construct large prime numbers and considerable efforts are made in order to search for them, such as the \emph{Great Internet Mersenne Prime Search} collaborative project \cite{Mersenne}.

Here we recall that the Lenstra--Pomerance--Wagstaff heuristic \cite{Wagstaff} suggests that the number of Mersenne primes $p=2^q-1$ in the interval $[1,x]$ is asymptotic to
$$
\frac{e^{\gamma}}{\log 2}\log \log x
$$
where $\gamma$ is Euler's constant. This gives a more precise form of Conjecture \ref{ConjMersenne}.


\section{Acknowledgments}
We thank Cec\'ilia Salgado for answering some questions on the relevant literature.

The first author was supported by ANID Doctorado Nacional 21190304. The second author was supported by ANID (ex CONICYT) FONDECYT Regular grant 1190442 from Chile. 


\end{document}